\documentclass[a4paper,12pt]{amsart}

\usepackage{amsmath}
\usepackage{amsthm}
\usepackage{amsfonts}
\usepackage{amsxtra}
\usepackage{amssymb}
\usepackage{mathdots}
\usepackage{graphicx}
\usepackage{mathrsfs}
\usepackage[pdftex]{color}
    \definecolor{MyRed}{rgb}{0.9,0,0}
    \definecolor{MyGreen}{rgb}{0,0.9,0}
    \definecolor{MyBlue}{rgb}{0,0,0.9}
\usepackage[pdftex]{hyperref}
\usepackage{hypcap}
\hypersetup{breaklinks=true,colorlinks=true,citecolor=MyGreen,urlcolor=MyRed,linkcolor=MyBlue,pdfpagemode=none}

\theoremstyle{plain}

\newtheorem{theorem}{Theorem}
\newtheorem*{theorem*}{Theorem}
\newtheorem*{Talagrand}{Talagrand's Inequality}

\newtheorem{lemma}[theorem]{Lemma}
\newtheorem{corollary}[theorem]{Corollary}

\theoremstyle{definition}

\newtheorem*{eg}{Example}

\def\Bollobas{Bollob{\'a}s}
\def\Luczak{{\L}uczak}
\def\Markstrom{Markstr\"om}
\def\Rucinski{Ruci{\'n}ski}
\def\Umea{Ume\aa}
\def\Malostranske{Malostransk\'{e}}
\def\Namesti{N\'am\v{e}st\'{i}}

\newcommand{\whp}{{\bf whp}}

\renewcommand{\leq}{\leqslant}
\renewcommand{\geq}{\geqslant}

\def\vs{\vspace{10pt}}

\newcommand{\E}{\bb E}

\newcommand{\bb}[1]{{\text{$\mathbb{#1}$}}}    
\newcommand{\bi}[1]{{\bf \em #1}}                

\def\Set#1{\left\{ #1 \right\}}
\def\Norm#1{\left\| #1 \right\|}

\newcommand{\sectionref}[1]{\hyperref[#1]{Section~\ref*{#1}}}
\newcommand{\subsectionref}[1]{\hyperref[#1]{Subsection~\ref*{#1}}}
\newcommand{\lemmaref}[1]{\hyperref[#1]{Lemma~\ref*{#1}}}
\newcommand{\corref}[1]{\hyperref[#1]{Corollary~\ref*{#1}}}

\begin{document}

\title{Random Latin square graphs}
\author{Demetres Christofides and Klas \Markstrom\/}
\date{\today}
\begin{abstract}
In this paper we introduce new models of random graphs, arising from
Latin squares which include random Cayley graphs as a special case.
We investigate some properties of these graphs including their
clique, independence and chromatic numbers, their expansion
properties as well as their connectivity and Hamiltonicity. The
results obtained are compared with other models of random graphs and
several similarities and differences are pointed out. For many
properties our results for the general case are as strong as the
known results for random Cayley graphs and sometimes improve the
previously best results for the Cayley case.
\end{abstract}
\maketitle

\section{Introduction}

The concept of random graphs is a very important notion in combinatorics. Although there are several models of random graphs, by a random graph one usually refers to the model $\mathscr{G}(n,p)$, the probability space of all graphs on $[n]$ in which every edge appears independently with probability $p$. For standard results on random graphs we refer the reader to the textbooks of \Bollobas\/~\cite{Bollobas01} and Janson, \Luczak\ and \Rucinski\/~\cite{Janson&Luczak&Rucinski00}.\vs

In this paper, we introduce new models of random graphs and study some of their properties with particular interest in their relation to the model $\mathscr{G}(n,p)$. Our models arise from Latin squares. Given a group, one can obtain Latin squares by considering its multiplication table or its division table. It turns out that the random graph obtained by the division table of a group $G$, is exactly the random Cayley graph of $G$ (with respect to a random subset $S$ of $G$.)\vs

Before defining our models, let us recall that a \bi{Latin square} of order $n$ is an $n \times n$ matrix $L$ with entries from a set of $n$ elements, such that in each row and in each column, every element appears exactly once. Given a Latin square $L$ with entries in a set $A$ of size $n$, and a subset $S$ of $A$, we define the \bi{Latin square graph} $G(L,S)$ on vertex set $[n]$, by joining $i$ to $j$ if and only if either $L_{ij} \in S$, or $L_{ji} \in S$.\vs

Suppose we are given a sequence $(L_n)$ of Latin squares of order $n$, with entries in $[n]$, say. Choosing $S \subseteq [n]$ by picking its elements independently at random with probability $p$, we obtain a random Latin square graph $G(L_n,S)$. We denote this model of random Latin square graphs by $\mathscr{G}(L_n,p)$. A related model is obtained by choosing a multiset $S$ of $k$ elements of $[n]$ by picking its elements independently and unifomly at random (with replacement). We denote this model by $\mathscr{G}(L_n,k)$. Note that our underlying graphs are simple. However, for the model $\mathscr{G}(L_n,k)$ it will be convenient for some of our results to retain multiple edges and loops. When we do this, we will denote this new model by $\mathscr{G}_m(L_n,k)$. To be more explicit, in this model the number of edges joining $i$ to $j$ is exactly the total number of times that $L_{ij}$ and $L_{ji}$ appear in $S$. In particular, every $G \in \mathscr{G}_m(L_n,k)$ is a $2k$-regular multigraph.\vs

A similar model is obtained by looking at the complement of the graph $G \in \mathscr{G}(L_n,p)$. We denote this model by $\bar{\mathscr{G}}(L_n,p)$. In general, this model is not the same as $\mathscr{G}(L_n,1-p)$, the reason being that $L_{ij}$ is not necessarily equal to $L_{ji}$. However, usually it is not too difficult to translate results from one model into the other, so we will only concentrate on $G \in \mathscr{G}(L_n,p)$.\vs

Note that, as mentioned above, our models include random Cayley
graphs as a special case. Indeed, given a group $G$, consider the
Latin square $L$ defined by $L_{xy} = xy^{-1}$. Then, given any
subset $S$ of elements of $G$, the Latin square graph $G(L,S)$ is
exactly the Cayley graph of $G$ with respect to $S$. The
multiplication table of a group is also a Latin square, giving rise
to what is usually known (motivated by the abelian case) as a
\bi{Cayley sum graph}. So this model includes random Cayley sum
graphs as well.\vs

We should mention here that there are several differences between random Cayley graphs and our more general models of random Latin square graphs. For example, random Cayley graphs are always vertex transitive. On the other hand a random Latin square graph, even if it arises from the multiplication table of a (non-abelian) group, might not even be regular. However, it is easy to see that random Latin square graphs are not far from being regular in the sense that the ratio of maximum to minimum degree is bounded above by 2.\vs

The fact that random Latin square graphs are almost regular (in the above sense) motivates also the comparison of our models with $\mathscr{G}_{n,r}$, the probability space of all $r$-regular graphs on $n$ vertices taken with the uniform measure. (As usual, it is always assumed that $rn$ is even.)\vs

Sometimes, it is easier to work with random Cayley graphs or random Cayley sum graphs for abelian groups, rather than random Latin square graphs. This is because we always have $L_{ij} = L_{ij}^{-1}$ in the case of Cayley graphs, and $L_{ij} = L_{ji}$ in the case of Cayley sum graphs, and so dependences between the edges can be easier to deal with. This sometimes leads to sharper results for the first two families of random graphs than for general random Latin square graphs; however we have opted to state our results only in the general case of random Latin square graphs.\vs

It seems that the general class of random graphs arising from Latin squares have not been studied before. However there has been much interest in random Cayley graphs and random Cayley sum graphs. For example, Agarwal, Alon, Aronov and Suri~\cite{Agarwal&Alon&Aronov&Suri94} established an upper bound on the clique number of random Cayley graphs arising from cyclic groups and used it to construct visibility graphs of line segments in the plane which need many cliques and complete bipartite graphs to represent them. In their study of a communication problem, Alon and Orlitsky~\cite{Alon&Orlitsky95} proved a similar upper bound for random Cayley graphs arising from abelian groups of odd order. Green~\cite{Green05}, using number theoretic tools, studied the clique number of various Cayley sum graphs and showed that some of them are good examples of Ramsey graphs while others are not. The diameter of random Cayley graphs with logarithmic degree was studied by Alon, Barak and Manber in~\cite{Alon&Barak&Manber87}. Alon and Roichman~\cite{Alon&Roichman94} proved that random Cayley graphs (on sufficiently many generators) are almost surely expanders, a result which was later improved by
several authors~\cite{Landau&Russell04,Loh&Schulman04,Christofides&Markstrom08}. The fact that random Cayley graphs are expanders has several consequences for the diameter, connectivity and Hamiltonicity of such graphs. Finally, some other aspects of the diameter, connectivity and Hamiltonicity of random Cayley graphs and random Cayley digraphs were studied in~\cite{Meng&Liu97,Meng&Huang98,Meng97,Meng&Huang96}.\vs

In this paper we extend many of these resutls to the general case of random Latin square graphs and show that the structure of the Latin squares have a non-trivial influence on many properties of random Latin square graphs. In \sectionref{results} we state and discuss our main results regarding random Latin square graphs. We prove these results in \sectionref{cliques}, \sectionref{colouring} and \sectionref{S:Expansion}. In \sectionref{Final} we give further examples and open problems.

\section{Statements and discussion of the results}\label{results}

In this section, we list our main results and make a few comments about them, comparing them with the corresponding results in the $\mathscr{G}(n,p)$ and $\mathscr{G}_{n,r}$ models. In \subsectionref{SS:2.1}, we will be interested in the maximum size of cliques and independent sets in $\mathscr{G}_{n,p}$, as well as the chromatic number of $\mathscr{G}_{n,p}$ and its complement. In \subsectionref{SS:2.2}, we will be interested in the expansion properties of random Latin square graphs as well as several consequences of these properties regarding connectivity and Hamiltonicity. For the results of this subsection it will be easier to work in the the models $\mathscr{G}_m(L_n,k)$ and $\mathscr{G}(L_n,k)$.\vs

\subsection{Cliques, independent sets and colouring}\label{SS:2.1}

We begin with an upper bound on the clique number of random Latin square graphs. It is well known that the clique number of $\mathscr{G}(n,1/2)$, is \whp\ asymptotic to $2\log_{2}{n}$. For the case of dense random regular graphs, it was proved in~\cite{Krivelevich&Sudakov&Vu&Wormald01} that the clique number of $\mathscr{G}_{n,n/2}$ is \whp\ asymptotic to $2\log_{2}{n}$.\vs

Guided by the above results, one might hope to prove that the clique number of $\mathscr{G}(L_n,1/2)$ is \whp\ $\Theta(\log{n})$. However, it turns out that this is not the case. Green~\cite{Green05} proved that the clique number of the random Cayley sum graph on $\bb Z_2^m$, with $p = 1/2$, is \whp\ $\Theta(\log{n} \log{\log{n}})$, where $n = 2^m = |\bb Z_2^m|$. In the same paper, Green proved that the clique number of the random Cayley sum graph on $\bb Z_n$, with $p = 1/2$, is \whp\ $\Theta(\log{n})$. This shows that, in general, results about the model $\mathscr{G}(L_n,p)$ can depend on the actual sequence of Latin squares chosen.\vs

To the best of our knowledge, the best known general result on the clique number is due to Alon and Orlitsky~\cite{Alon&Orlitsky95}, which says that the clique number of a random Cayley graph arising from an abelian group of odd order $n$ is \whp\ $O((\log{n})^2)$. Using similar methods, we have managed to show that the same bound is in fact true for random Latin square graphs. In particular, it is also true for random Cayley graphs arising from non-abelian groups. We believe but cannot prove that the $2$ in the exponent can be reduced further

\begin{theorem}[Clique number; upper bound]\label{clique-upper}
Let $0 < p < 1$ be a fixed constant and let $d = 1/(2p-p^2)$. Then, for almost every $G \in \mathscr{G}(L_n,p)$, we have
\[ \omega(G) \leq  27 \left(\log_d{n} \right)^2.\]
\end{theorem}

Since the model $\bar{\mathscr{G}}(L_n,p)$ is different from $\mathscr{G}(L_n,1-p)$, we cannot immediately deduce a corresponding upper bound for the independence number. One way to find such a bound is to couple the model $\bar{\mathscr{G}}(L_n,p)$ with $\mathscr{G}(L_n,1-p)$, and use \autoref{clique-upper} to deduce that for almost every $G \in \mathscr{G}(L_n,p)$,
\[ \alpha(G) = \omega(\bar{G}) \leq  27 \left(\log_{1/(1-p^2)}{n} \right)^2. \]

In fact, using an argument similar to the one used in the proof of \autoref{clique-upper}, we can obtain a slightly better result.

\begin{theorem}[Independence number; upper bound]\label{independence-upper}
Let $0 < p < 1$ be a fixed constant and let $d = 1/(1-p)$. Then, for almost every $G \in \mathscr{G}(L_n,p)$, we have

\[ \alpha(G) \leq 27 \left(\log_d{n} \right)^2.\]
\end{theorem}

Recall that the \bi{(vertex) clique cover number} $\theta(G)$ of a graph $G$ is the smallest integer $k$ such that the vertex set of $G$ can be partitioned into $k$ cliques. I.e. $\theta(G) = \chi(\bar G)$. So an immediate corollary of \autoref{clique-upper} is:

\begin{corollary}[Clique cover number; lower bound]\label{cc-lower}
Let $0 < p < 1$ be a fixed constant and let $d = 1/(2p-p^2)$. Then, for almost every $G \in \mathscr{G}(L_n,p)$, we have

\[ \theta(G) \geq \frac{n}{27 \left(\log_d{n} \right)^2}. \qed\]
\end{corollary}

Similarly, \autoref{independence-upper} implies:

\begin{corollary}[Chromatic number; lower bound]\label{chromatic-lower}
Let $0 < p < 1$ be a fixed constant and let $d = 1/(1-p)$. Then, for almost every $G \in \mathscr{G}(L_n,p)$, we have
\[ \chi(G) \geq \frac{n}{27 \left(\log_d{n} \right)^2}. \qed\]
\end{corollary}

We now move to our upper bound on the chromatic number of random Latin square graphs. Recall that for constant $p$, the chromatic number of $\mathscr{G}(n,p)$ is \whp\ asymptotic to $\frac{n}{2 \log_{b}{n}}$, where $b = 1/(1-p)$. A similar behaviour was proved in~\cite{Krivelevich&Sudakov&Vu&Wormald01} for the case of random regular graphs of high degree. More specifically, it was proved that for any $\varepsilon > 0$, if $\varepsilon n \leq r \leq 0.9n$, then the chromatic number of $\mathscr{G}_{n,r}$ is \whp\ asymptotic to $\frac{n}{2 \log_{b}{n}}$, where $b = n/(n-r)$.\vs

For the case of random Latin square graphs, we prove an upper bound
of the same order of magnitude. However, since our lower bound is
only of order $\frac{n}{\left( \log_{b}{n} \right)^2}$, we still do
not have a sharp asymptotic result for the chromatic number. In
fact, as in the case of the clique and independence numbers, we know
that the chromatic number can depend on the sequence of Latin
squares chosen. For example, the result of Green~\cite{Green05}
mentioned above, that the independence number of the random Cayley
sum graph on $\bb Z_n$ (with $p = 1/2$) is \whp\ $\Theta(\log{n})$,
provides a lower bound for the chromatic number of these graphs
which is of the same order of magnitude as our corresponding upper
bound. On the other hand, we claim that the chromatic number of the
random Cayley sum graph on $\bb Z_2^m$ is \whp\
$\Theta(\frac{n}{\log{n} \log{\log{n}}})$, where $n = 2^m = |\bb
Z_2^m|$. The lower bound follows immediately from the result of
Green~\cite{Green05} mentioned above for the independence number of
these graphs. The upper bound does not follow directly from that
result, however it follows from its proof in~\cite{Green05} that in
fact there is \whp\ a $\left\lfloor \log{m} + \log\log{m} - 1
\right\rfloor$-dimensional subspace of $\bb Z_2^m$ which is an
independent set. Indeed, given this result, it follows that \whp, a
random Cayley sum graph on $\bb Z_2^m$ can be partitioned into at
most $\frac{4n}{\log{n}\log{\log{n}}}$ independent sets of this
form.\vs

In fact, our upper bound on the chromatic number will be an immediate consequence of an upper bound on the list-chromatic number. Recall that the \bi{list-chromatic number} $\chi_l(G)$ of a graph $G$ is the smallest positive integer $k$ such that for any assignment of $k$-element sets $L(v)$ to the vertices of $G$, there is a proper vertex colouring $c$ of $G$ with $c(v) \in L(v)$ for every vertex $v$ of $G$.

\begin{theorem}[List-chromatic number; upper bound]\label{list-chromatic-upper}
Let $0 < p < 1$ be a fixed constant and let $d = 1/(1-p)$. Then, for almost every $G \in \mathscr{G}(L_n,p)$, we have
\[ \chi_l(G) \leq \frac{n}{\frac{1}{4}\log_d{n} - \frac{1}{2}\log_d{\log_d{n}} - 2}. \]
\end{theorem}

\begin{corollary}[Chromatic number; upper bound]\label{chromatic-upper}
Let $0 < p < 1$ be a fixed constant and let $d = 1/(1-p)$. Then, for almost every $G \in \mathscr{G}(L_n,p)$, we have
\[ \chi(G) \leq \frac{n}{\frac{1}{4}\log_d{n} - \frac{1}{2}\log_d{\log_d{n}} - 2}. \qed \]
\end{corollary}

With similar methods we will show the following upper bound for the clique cover number.

\begin{theorem}[Clique cover number; upper bound]\label{cc-upper}
Let $0 < p < 1$ be a fixed constant and let $d = 1/p$. Then, for almost every $G \in \mathscr{G}(L_n,p)$, we have
\[ \theta(G) \leq \frac{n}{\frac{1}{2}\log_d{n} - \log{\log_d n} - 6}.\]
\end{theorem}

From \autoref{chromatic-upper} and \autoref{cc-upper} we deduce corresponding lower bounds on the independence and clique numbers.

\begin{corollary}[Independence number; lower bound]\label{independence-lower}
Let $0 < p < 1$ be a fixed constant and let $d = 1/(1-p)$. Then, for almost every $G \in \mathscr{G}(L_n,p)$, we have
\[ \alpha(G) \geq \frac{1}{2}\log_d{n} - \log{\log_d n} - 6.  \qed \]
\end{corollary}

\begin{corollary}[Clique number; lower bound]\label{clique-lower}
Let $0 < p < 1$ be a fixed constant and let $d = 1/p$. Then, for almost every $G \in \mathscr{G}(L_n,p)$, we have
\[ \omega(G) \geq \frac{1}{4}\log_d{n} - \frac{1}{2}\log_d{\log_d{n}} - 2. \qed\]
\end{corollary}

\subsection{Expansion and related properties}\label{SS:2.2}

Alon and Roichman~\cite{Alon&Roichman94} proved that random Cayley graphs on logarithmic number of generators are expanders \whp. Our main result of this subsection states that a similar result holds in the case of random Latin square graphs.\vs

Before stating our result we need to introduce some notation. Given a multigraph $G$, its \bi{adjacency matrix} is the 0,1 matrix $A=A(G)$ with rows and columns indexed by the vertices of $G$, in which $A_{xy}$ is the number of edges in $G$ joining $x$ to $y$. If $G$ is $d$-regular then its \bi{normalised adjacency matrix} $T = T(G)$ is defined by $T = \frac{1}{d}A$. Note that $T$ is a real symmetric matrix, so it has an orthonormal basis of real eigenvectors. We will write $\lambda_0 \geq \lambda_1 \geq \dots \geq \lambda_{n-1}$ for the eigenvalues of $T$. It is easy to check that $\lambda_0 = 1$ and that $\lambda_{n-1} \geq -1$. We will write $\mu$ for the second largest eigenvalue in absolute value, i.e.~$\mu = \max{\{|\lambda_1|,|\lambda_{n-1}|\}}$.\vs

Finally, for $0 < x < 1$, we define
\[ H(x) = x\log{(2x)} + (1-x)\log{(2(1-x))}, \]
where we use the convention that all logarithms are natural.\vs

We can now state our main theorem.

\begin{theorem}[Second eigenvalue]\label{second eigenvalue}
Let $L$ be an $n \times n$ Latin square with entries in $[n]$ and let $G \in \mathscr{G}_m(L,k)$. Then, for every $0 < \varepsilon < 1$,
\[ \Pr(\mu(G) \geq \varepsilon) \leq 2n \exp{\Set{-k H\left(\frac{1 + \varepsilon}{2} \right)}} \leq 2n \exp{\Set{- \frac{k \varepsilon^2}{2} }}. \]
\end{theorem}

We remark that if $L$ is the difference table of a group, then the above theorem is similar to the result of Alon and Roichman mentioned in the beginning of this subsection. The only difference is that the bounds appearing in the above theorem, are the same as the bounds appearing in the authors' proof~\cite{Christofides&Markstrom08} of the Alon-Roichman theorem and are slightly better than the original bounds of the Alon-Roichman theorem.\vs

Recall that a graph $G$ is an $(n,d,\varepsilon)$-\bi{expander} if it is a graph on $n$ vertices with maximum degree $d$ such that for every subset $W$ of its vertices of size at most $n/2$ we have $|N(W) \setminus W| \geq \varepsilon |W|$, where $N(W)$ denotes the neighbourhood of $W$. Note that for this definition we may ignore any multiple edges or loops that $G$ may have. For more on expander graphs and their applications, we refer the reader to the recent survey of Hoory, Linial and Wigderson~\cite{Hoory&Linial&Wigderson06}.\vs

It is well known~\cite{Tanner84,Alon&Milman85} that a small second eigenvalue implies good expansion properties. The following corollary is an immediate consequence of \autoref{second eigenvalue} together with this fact.

\begin{corollary}[Expansion]\label{expansion}
For every $\delta> 0$, there is a $c(\delta)>0$ depending only on $\delta$, such that almost every $G \in \mathscr{G}_m(L_n,c(\delta) \log{n})$ is an $(n,2c(\delta)\log{n},\delta)$-expander.\qed
\end{corollary}

The fact that the second eigenvalue of the graph is small implies that such a graph has several properties that many `random-like' graphs possess. Informally, a graph of density $p$ is pseudorandom if its edge distribution resembles the edge distribution of $\mathscr{G}(n,p)$. The study of pseudorandom graphs was initiated by Thomason in~\cite{Thomason87a,Thomason87b}. Chung, Graham and Wilson~\cite{Chung&Graham&Wilson89} showed that many properties that a graph may possess, including the property of having small second eigenvalue, are in some sense equivalent to pseudorandomness.\vs

Here we list just a few of these consequences, mostly taken from the recent survey of Krivelevich and Sudakov~\cite{Krivelevich&Sudakov06}. We omit some of the proofs, but we note that some care needs to be taken since our graphs are multigraphs, while the result in the survey are stated only for simple graphs.\vs

To begin with, let us consider what value of $k$ guarantees that almost every $G \in \mathscr{G}(L_n,k)$ is connected. Let us first recall the corresponding results in $\mathscr{G}(n,p)$ and $\mathscr{G}_{n,r}$. It is well known that for any fixed $\delta> 0$, if $p \leq (1 - \delta) \log{n}/n$, then $\mathscr{G}(n,p)$ is \whp\ disconnected, while if $p \geq (1 + \delta) \log{n}/n$, then $\mathscr{G}(n,p)$ is \whp\ connected. On the other hand, $\mathscr{G}_{n,r}$ is \whp\ connected provided that $r \geq 3$.\vs

So what is the right threshold for the connectivity of random Latin square graphs? Once again this depends on the sequence $(L_n)$ of Latin squares chosen. For example, the Cayley graph of $\bb Z_q$ for $q$ prime, with respect to any set $S$ containing a non-trivial element is connected. On the other hand, the Cayley graph of $G = \bb Z_2^m$ with respect to any set of size less than $m = \log_2{|G|}$ is disconnected. Here, we prove that choosing slightly more elements are enough to guarantee \whp\ the connectedness not only of the random Cayley graph of $\mathbb{Z}_2^m$ but in fact the connectedness of any random Latin square graph. 

\begin{theorem}[Connectedness]\label{connectedness}
For any fixed $\delta> 0$, almost every $G \in \mathscr{G}(L_n,(1+\delta)\log_2{n})$ is \whp\ connected.
\end{theorem}

\begin{proof}
It is enough to prove the result for $0 < \delta< 1/2$. Note that $H(x)$ is continuous in $(0,1)$ and tends to $\log{2}$ as $x$ tends to 1. Pick an $x$ such that $H(x) \geq (1 - \delta/2)\log{2}$. Then, for $k = |S| = (1 + \delta)\log_2{n}$, we have $kH(x) \geq (1 + \delta/4)\log{n}$. Thus,
\[ \Pr(\mu(G(L,S)) \geq 2x - 1) \leq 2n \exp{\left\{-(1 + \delta/4)\log{n}\right\}} = 2n^{-\delta/4} = o(1).\]
Thus \whp, $\mu(G(L,S)) < 2x - 1 < 1$. It is well known that if $\mu(G) < 1$ then $G$ is connected, so the result follows.
\end{proof}

Let us now move to the vertex connectivity of random Latin square graphs. Recall that the \bi{vertex connectivity} $\kappa(G)$ of a graph $G$ is the minimal number of vertices that we need to remove in order to disconnect $G$. Clearly the vertex connectivity of any graph is at most its minimum degree $\delta(G)$. It is well known that for $G \in \mathscr{G}(n,p)$ we have $\kappa(G) = \delta(G)$. Recently, it was shown in~\cite{Krivelevich&Sudakov&Vu&Wormald01,Cooper&Frieze&Reed02} that the same holds for random $r$-regular graphs provided $3 \leq r \leq n-4$. In our case, the Cayley graphs on $\mathbb{Z}_2^m$ show that no such result can hold if the generating set $S$ has size less than $\log_2{n}$. Can we expect that such a result holds if the size of $S$ is large enough? As the following example shows the answer is no. To understand the idea of the example, observe that if we can find two neighbouring vertices $x,y$ in a $d$-regular graph $G$ with exactly $d-1$ common neighbours, then removing these neighbours disconnects $x$ and $y$ from the other vertices of $G$ and thus $\kappa(G) \leqslant d-1$. While in a random $d$-regular graph this is very unlikely to happen, in the specific example that follows we can (deterministically) guarantee that the random Latin square graph is regular and furthermore its vertex set can be partitioned into pairs so that vertices in the same pair have exactly the same neighbours outside of this pair. It thus only remains to check that \whp\ at least two vertices which belong to the same pair will be adjacent.\vs

\begin{eg}
Define a Latin square $L$ on $\{0,1,\ldots,r-1\} \times \{0,1\}$ with entries in $\{0,1,\ldots,2r-1\}$ as follows:
\begin{itemize}
\item[] $L_{(x,0),(y,0)} = \begin{cases}
        x + y & \text{ if } x \leq y \\
        x + y + r & \text{ if } x > y
        \end{cases}$
\item[] $L_{(x,0),(y,1)} = \begin{cases}
        x + y + r& \text{ if } x \leq y \\
        x + y & \text{ if } x > y
        \end{cases}$
\item[] $L_{(x,1),(y,0)} = \begin{cases}
        x + y + r& \text{ if } x \leq y \\
        x + y & \text{ if } x > y
        \end{cases}$
\item[] $L_{(x,1),(y,1)} = \begin{cases}
        x + y & \text{ if } x \leq y \\
        x + y + r& \text{ if } x > y
        \end{cases}$
\end{itemize}
Here, addition is done modulo $2r$. It can be easily checked that $L$ is indeed a Latin square. Pick any $S \subseteq \{0,1,\ldots,2r-1\}$ and let $G = G(L,S)$. Note that $G$ is $d$-regular for some $d$. Note also that for any $x \in \{0,1,\ldots,r-1\}$, we have that $N_G((x,0)) \setminus \{(x,1)\} = N_G((x,1)) \setminus \{(x,0)\}$, where $N_G$ denotes the neigbourhood of a vertex in $G$. But then, if $(x,0)$ is adjacent to $(x,1)$ for some $x$, and $G$ is not complete, we have that $\kappa(G) \leq d-1$. Indeed, $N_G((x,0)) \setminus \{(x,1)\}$ is a disconnecting set of size $d-1$. Now $(x,0)$ is adjacent to $(x,1)$ if and only if $2x + r \in S$. Let $p = p(r) \in (0,1)$ be chosen such that $pr \to \infty$ and $(1-p)r \to \infty$ as $r \to \infty$ and choose $S$ by picking its elements independently at random with probability $p$. Then \whp\ $G$ is not complete and there is an $x$ such that $(x,0)$ is adjacent to $(x,1)$  and so $\kappa(G) \leq \delta(G)-1$.
\end{eg}

The above example shows that even if the size of $S$ is large enough the vertex connectivity of a random Latin square qraph can be \whp\ strictly smaller than its minimum degree. However, our next theorem shows that if $S$ is large enough then the vertex connectivity of a random Latin square graph is \whp\ at most one less than its minimum degree.

\begin{theorem}[Vertex connectivity]\label{vertex-connectivity}
There is an absolute constant $C \leq 168$ such that whenever
$C\log{n} \leq k \leq n/4$, then $\delta(G) - 1 \leq \kappa(G) \leq
\delta(G)$ for almost every $G \in \mathscr{G}(L_n,k)$.
\end{theorem}

The example of $\mathbb{Z}_2^m$ shows that we cannot take $C$ to be
equal to 1. It would be interesting to know whether every $C$
strictly larger than 1 works or not. It seems that our proof cannot
bring the value of $C$ down to $1+ \delta$ for any $\delta>0$, so we have
not tried to optimize the value of $C$ that our proof gives.\vs

It should be noted that above result is not a direct consequence of the expansion properties of random Latin square graphs. From \autoref{second eigenvalue}, we can only deduce that $\mu = O(\sqrt{\log{n}/k})$. However one can construct examples of $d$-regular graphs on $n$ vertices, with $d = \Omega(\log{n})$, $\mu = \Omega(\sqrt{\log{n}/d})$ but $\kappa(G) \leq d -  \Omega(\log{n})$. We refer the reader to the discussion following~\cite[Theorem 4.1]{Krivelevich&Sudakov06} for more details about how one can construct such a graph.\vs

Similar to the vertex connectivity, the \bi{edge connectivity} $\lambda(G)$ of a graph $G$ is the minimal number of edges that we need to remove in order to disconnect $G$. It is easy to show that $\kappa(G) \leq \lambda(G) \leq \delta(G)$. Hence, \autoref{vertex-connectivity} applies with $\kappa(G)$ replaced by $\lambda(G)$. In fact, our next theorem shows that we can do a bit more. If $|S| \geq (1 + \delta)\log_2{n}$ then $\whp$ the edge connectivity is equal to the minimum degree of $G$. In view of random Cayley graphs on $\mathbb{Z}_2^m$, this is in fact best possible.

\begin{theorem}[Edge connectivity]\label{edge-connectivity}
For any $\delta> 0$, if $L$ is an $n \times n$ Latin square with entries in $[n]$ and $S$ is a set of $(1+\delta)\log_2{n}$ elements of $[n]$, chosen independently and uniformly at random, then \whp, $\lambda(G(L,S)) = \delta(G(L,S))$.
\end{theorem}

Another graph property which follows from pseudorandomness is that of Hamiltonicity. Again, this property depends on the structure of the Latin square. For example, the Cayley graph of $\mathbb{Z}_q$ for $q$ prime, with respect to any non-trivial element is Hamiltonian. On the other hand, as it was mentioned earlier, the Cayley graph of $G = \bb Z_2^m$ with respect to any set of size less than $m = \log_2{|G|}$ is not even connected. A very appealing conjecture attributed to Lov\'asz, states that every connected Cayley graph is Hamiltonian. Together with \autoref{connectedness}, this would imply for example that every random Cayley graph on $(1 + o(1))\log_2{n}$ generators is Hamiltonian. Pak (see~\cite{Pak&Radoicic09}) conjectured that there is a constant $c \geqslant 1$ such that every random Cayley graph on $(c + o(1))\log_2{n}$ generators is Hamiltonian. However, even this consequence is still not known. Recently, Krivelevich and Sudakov~\cite{Krivelevich&Sudakov03} proved that every $d$-regular graph on $n$ vertices satisfying
\[ \mu \leq \frac{(\log{\log{n}})^2}{1000\log{n}(\log{\log{\log{n}}})}, \]
is Hamiltonian, provided $n$ is large enough. Using this, together with the proof technique of the Alon-Roichman theorem, they proved that a random Cayley graph on $O((\log{n})^5)$ generators is $\whp$ Hamiltonian. Here, we extend this result to random Latin square graphs as well. Moreover, using \autoref{second eigenvalue} directly, we can in fact replace $(\log{n})^5$ by $(\log{n})^3$.

\begin{theorem}[Hamiltonicity]\label{Hamiltonicity}
If $k = \omega\left(\frac{(\log{n})^3 (\log{\log{\log{n}}})^2}{(\log{\log{n}})^2}\right)$, then almost every $G \in \mathscr{G}(L_n,k)$ is Hamiltonian.\footnote{Recall that $f(n) = \omega(g(n))$ means that $f(n)/g(n) \to \infty$ as $n \to \infty$.}
\end{theorem}

\section{Cliques and independent sets}\label{cliques}

We begin by finding upper bounds for the clique number of random Latin square graphs. Naturally, one would like to find a good upper bound for the expected number of $d$-cliques of a random Latin square graph, and from this deduce a corresponding upper bound for the clique number. Given $A \subseteq [n]$ let $A' = \{L_{ij}: i,j \in A, i \neq j\}$. If $|A| = d$, then $|A'|$ can be as large as $\binom{d}{2}$ and as small as $d-1$. In the former case, the probability that $A$ forms a clique in $\mathscr{G}(L_n,p)$ is $(2p-p^2)^{\binom{d}{2}}$. However, in the latter case, this probability is at least $p^{d-1}$. So, unless one is able to bound the number of $A \subseteq [n]$ for which $|A'|$ is relatively small, then this approach cannot give any good bounds. Our approach will be to show that any $A \subseteq [n]$ of size $d$, has a subset $B$ of size $\Omega(\sqrt{d})$, such that $|B'|$ is relatively large, i.e. $\Omega(|B|^2)$. By standard arguments it will then follow that \whp\, (if $d$ is large enough,) no such $B$ forms a clique, and hence no $A \subseteq [n]$ of size $d$ forms a clique. Before stating our main lemma, we need to introduce some more notation.
\begin{itemize}
\item[] $n_2(A) = |\{ \{i,j\}: i,j \in A \text{ distinct and } L_{ij} = L_{ji} \}|$;
\item[] $ n_3(A) = |\{(i,j,k): i,j,k \in A \text{ distinct and } L_{ij} = L_{jk} \}|$;
\item[] $n_4(A) = |\{ \{(i,j),(k,l)\}: i,j,k,l \in A \text{ distinct and } L_{ij} = L_{kl} \}|$.
\end{itemize}
If $x \in A'$ appears exactly $r_x$ times as $L_{ij}$ for distinct $i,j \in A$, then, with the above notation, we have
\[ n_2(A) + n_3(A) + n_4(A) = \sum_{x \in A'}\binom{r_x}{2}.\]

We are now ready to state and prove our main lemma.

\begin{lemma}\label{MainLemma}
Let $A$ be a set of elements of $X$ of size $a$. Then for every $b \leq a$, $A$ contains a subset $B$ of size $b$ such that
\[ |B'| \geq b(b-1) \left(1 - \frac{b-2}{a-2} - \frac{(b-2)(b-3)}{2(a-3)} \right) - n_2(B). \]
\end{lemma}

\begin{proof}
For any $B \subseteq A$ of size $b$, we have
\begin{align*}
|B'| &= b(b-1) - \sum_{x \in B'}(r_x - 1) \\
& \geq b(b-1) - \sum_{x \in B'}\binom{r_x}{2} \\
&= b(b-1) - n_2(B) - n_3(B) - n_4(B).
\end{align*}
Picking $B$ at random from all $b$ element subsets of $A$, we have
\begin{itemize}
\item[] $\E (n_3(B)) = n_3(A) \frac{b(b-1)(b-2)}{a(a-1)(a-2)}$; and
\item[] $\E (n_4(B)) = n_4(A) \frac{b(b-1)(b-2)(b-3)}{a(a-1)(a-2)(a-3)}$.
\end{itemize}
Fixing distinct $i,j \in A$, there is exactly one $k \in [n]$ such that $L_{ij} = L_{jk}$, hence $n_3(A) \leq a(a-1)$. Similarly, fixing distinct $i,j,k \in A$, there is exactly one $L \in [n]$ such that $L_{ij} = L_{kl}$, hence $n_4(A) \leq  \frac{a(a-1)(a-2)}{2}$. It follows that
\[ \E (|B'| + n_2(B)) \geq b(b-1)\left(1 - \frac{b-2}{a-2} - \frac{(b-2)(b-3)}{2(a-3)} \right), \]
and hence there is a choice of $B$ satisfying the requirements of the lemma.
\end{proof}\vspace{1pt}

We can now prove \autoref{clique-upper}.\vspace{1pt}

\begin{proof}[Proof of \autoref{clique-upper}]
Let $d = 1/(2p - p^2)$, let $b = 3 \log_d{n}$ and let $a = 3b^2$. Pick any $A \subseteq [n]$ of size $a$. By \lemmaref{MainLemma}, there is a $B \subseteq A$ of size $b$, such that
\[|B'| \geq \frac{5}{6}b^2 - n_2(B) + O(b).\]
Pick $|B'|$ pairs $(i,j)$ in $B \times B$, with $i \neq j$, such that all $L_{ij}$ are distinct. Suppose that for exactly $k$ of the pairs we have $L_{ij} = L_{ji}$. It follows that there are at least
\[ (|B'| - k) - \left(\binom{b}{2} - n_2(B) \right) = \frac{1}{3}b^2 - k + O(b),\]
sets $\{i,j\}$, such that both $(i,j)$ and $(j,i)$ have been chosen (and so $L_{ij} \neq L_{ji}$). Therefore, the probability that $B$ is a clique is at most
\[ p^k \left( 2p - p^2 \right)^{\frac{1}{3}b^2 - k + O(b)} \leq \left( 2p - p^2 \right)^{\frac{1}{3}b^2 + O(b)}. \]
So the expected number of cliques $B \subseteq [n]$ of size $b$ with $|B'| \geq \frac{5}{6}b^2 - n_2(B) + O(b)$ is at most
\[ \binom{n}{b} \left( 2p - p^2 \right)^{\frac{1}{3}b^2 + O(b)} \leq  \frac{1}{b!} \left( n \left( 2p - p^2 \right)^{\frac{1}{3}b + O(1)} \right)^b = o(1). \]
Thus, by Markov's Inequality, we deduce that \whp, no such $B$ exists. By \autoref{MainLemma}, it now follows that \whp, there is no clique of size $3b^2$, as required.
\end{proof}\vspace{1pt}

In a similar way, we can prove the upper bound for the independence number.\vspace{1pt}

\begin{proof}[Proof of \autoref{independence-upper}]
Let $b = 3\log_{1/(1-p)}{n}$, and let $a = 3b^2$. Pick any $A \subseteq [n]$ of size $a$. By \lemmaref{MainLemma}, there is a $B \subseteq A$ of size $b$, such that
\[|B'| \geq \frac{5}{6}b^2 - n_2(B) + O(b) \geq \frac{1}{3}b^2 + O(b).\]
Therefore, the probability that $B$ is an independent set, is at most $(1-p)^{b^2/3 + O(b)}$, and so the expected number of independent sets is
\[ \binom{n}{b}(1-p)^{b^2/3 + O(b)} \leq  \frac{1}{b!} \left( n (1-p)^{b/3 + O(1)}\right)^b = o(1). \]
Thus, by Markov's Inequality, we deduce that \whp, no such $B$ exists. By \autoref{MainLemma} it now follows that \whp, there is no independent set of size $3b^2$, as required.
\end{proof}\vspace{1pt}

\section{Colouring}\label{colouring}

We now move to the proof of the upper bounds on the chromatic number. Before presenting our proof, let us see why a standard approach from the theory of random graphs does not seem to generalise in a straightforward manner.\vs

Suppose we could show that \whp, every induced subgraph of $G \in \mathscr{G}(L_n,1/2)$ on $n_1 = n/(\log{n})^2$ vertices has an independent set of size at least $s_1 = (2 - \varepsilon) \log_{2}{n}$. It then follows immediately that \whp, the chromatic number is at most $n/s_1 + n_1 \sim n/2 (\log_2{n})$. To do this, one usually shows that the probability that a given induced subgraph on $n_1$ vertices does not contain an independent set of size $s_1$ is $O(\exp{ \left\{ -n^{1 + \delta} \right\} })$, for some $\delta> 0$. However in our model, this is far from being true. In fact, the probability that $G \in \mathscr{G}(L_n,1/2)$ is empty is $2^{-n}$, which is much larger than $O(\exp{ \left\{ -n^{1 + \delta} \right\} })$. It turns out that this problem can be rectified by using the expansion properties of the graph $G$. We refer the reader to~\cite{Alon&Krivelevich&Sudakov99} to see how one can do this. Here, we will use a different approach from which we can obtain a better constant in the bound.\vs

Another approach for finding an upper bound for the chromatic number, is to analyse the greedy algorithm. This is the approach that we are going to use. This approach will in fact give an upper bound on the list-chromatic number as well. However, we need to modify the standard argument, because of the dependencies in the appearance of edges. In our modification we will make use of \hyperlink{Talagrand}{Talagrand's Inequality}~\cite{Talagrand95}. We will use the following version taken (essentially) from~\cite{Janson&Luczak&Rucinski00}.

\begin{Talagrand}\hypertarget{Talagrand}
Let $X$ be a non-negative integer valued random variable, not identically 0, which is determined by $n$ independent random variables and let $M$ be a median of $X$. Suppose also that there exist $K$ and $r$ such that
\begin{enumerate}
\item $X$ is $K$-Lipschitz. I.e.~changing the outcome of one of the variables, changes the value of $X$ by at most $K$.
\item For any $s$, if $X \geq s$, then there is a set of at most $rs$ of the variables, whose outcome certifies that $X \geq s$.
\end{enumerate}
Then
\[ \Pr(|X - M| \geq t) \leq \begin{cases}
4\exp{ \left\{ -\frac{t^2}{8 r K^2 M} \right\} } & \text{ if } 0 \leq t \leq M;\\
2\exp{ \left\{ -\frac{t}{8 r K^2} \right\} } & \text{ if } t > M.
\end{cases} \]
\end{Talagrand}

In particular, it follows that,
\begin{align*}
|\E X - M| &\leq \E|X-M| = \int_{0}^{\infty} \Pr(|X-M| > t)\; dt \\
           &\leq 4 \int_{0}^{M} \exp{ \left\{ -\frac{t^2}{8 r K^2 M} \right\} }\; dt + 2 \int_{M}^{\infty} \exp{ \left\{ -\frac{t}{8 r K^2 } \right\} }\; dt \\
           &\leq 2K \sqrt{8 \pi r M} + 16rK^2.
\end{align*}

Since also $M = 2M \Pr(X \geq M) \leq 2\E X$, we deduce that for $0 \leq t \leq \E X$
\[ \Pr \left( |X - \E X| \geq t + 16rK^2 + 16K \sqrt{r \E X} \right) \leq 4\exp{ \left\{ -\frac{t^2}{16 r K^2 \E X} \right\} }.\]

This is the form of \hyperlink{Talagrand}{Talagrand's Inequality} that we will be using.\vs

Let us now proceed to the proof of \autoref{list-chromatic-upper}.\vspace{1pt}

\begin{proof}[Proof of \autoref{list-chromatic-upper}]
Let $d = 1/(1-p)$ and let $u = \frac{1}{4}\log_d{n} - \frac{1}{2}\log_d{\log_d n} - 2$. Suppose every vertex $v$ has a list $L(v)$ of size $\lfloor n/u \rfloor$. Fix an ordering $v_1,\ldots, v_n$ of the vertices. Suppose we are given a (not necessarily proper) colouring $c$ of vertices $v_1,\ldots,v_m$, such that $c(v_i) \in L(v_i)$ for each $1 \leq i \leq m$. Suppose $L(v_{m+1}) = \{x_1,\ldots,x_{\lfloor n/u \rfloor}\}$, let $c_i = c_i(m)$ be the number of times that colour $x_i$ is used on vertices $v_1,\ldots,v_m$ and let $A_{m+1}$ be the event that $v_{m+1}$ has an earlier neighbour in every colour of the list $L(v_{m+1})$. We claim that $\Pr(A_{m+1}) = o(1/n)$. Having proved this, we proceed by list-colouring the graph greedily. The probability that this fails is at most $\sum_{m=1}^n\Pr(A_m) = o(1)$, so by Markov, we have \whp\ $\chi_l(G) \leq n/u$.\vs

To prove our claim, let $B_i = B_i(m)$ be the event that $v_{m+1}$ is joined with an earlier vertex of colour $x_i$. Then clearly $\Pr(B_i) \leq 1 - (1-p)^{2c_i}$. Let $Y$ be the number of colours in $L(v_{m+1})$ appearing on earlier neighbours of $v_{m+1}$. Then
\[ \E Y \leq \frac{n}{u} - \sum (1-p)^{2c_i} \leq \frac{n}{u} - \frac{n}{u} (1-p)^{2mu/n} \leq \frac{n}{u}\left(1 - (1-p)^{2u}\right), \]
where the second inequality follows from the Arithmetic-Geometric Mean Inequality. Let $X = Y - \E Y + \frac{n}{u}(1-(1-p)^{2u})$ and let $t = c\frac{n}{u}(1-p)^{2u}$, for some $0 < c < 1$ to be determined later. Then $X$ satisfies the conditions of \hyperlink{Talagrand}{Talagrand's Inequality} with $K = 2$ and $r = 1$. Note that, for $n$ large enough, $0 \leq t \leq \E X$, so
\begin{multline*}
\Pr\left(|X - \E X| \geq c\frac{n}{u}(1-p)^{2u} + 64 + 32 \sqrt{\frac{n}{u}(1 - (1-p)^{2u})} \right) \leq \\
4 \exp{\Set{-\frac{c^2 n (1-p)^{4u}}{64u} } } = 4 \exp{\Set{ -\frac{c^2 n \left( \log_{d}{n} \right)^2}{64u(1-p)^4} } } \leq \\ 4 \exp{\Set{- \frac{c^2 \log{n}}{16(1-p)^8 \log{(1/(1-p))} } }}.
\end{multline*}

By elementary calculus, it is easy to show that $16x^8 \log{(1/x)} \leq 2/e$ whenever $0 < x < 1$. Hence, choosing any $c$ with $\sqrt{2/e} < c < 1$, we deduce that
\[ \Pr\left(|X - \E X| \geq c\frac{n}{u}(1-p)^{2u} + 64 + 32 \sqrt{\frac{n}{u}(1 - (1-p)^{2u})} \right) = o(1/n).\]
In particular, since $Y \leq X$,
\[ \Pr\left(Y \geq \frac{n}{u} - c\frac{n}{u}(1-p)^{2u} + 64 + 32 \sqrt{\frac{n}{u}} \right) = o(1/n). \]
Since
\[ \frac{n}{u}(1-p)^{4u} = \frac{(\log_d{n})^2}{u(1-p)^8} \to \infty,\]
we deduce that (for $n$ large enough,)
\[ \Pr(A_{m+1}) = \Pr(Y \geq \lfloor n/u \rfloor) = o(1/n). \qedhere\]
\end{proof}\vspace{1pt}

Similarly, we can give an upper bound to the clique cover number.\vspace{1pt}

\begin{proof}[Proof of \autoref{cc-upper}]
Let $d = 1/p$ and let $u = \frac{1}{2}\log_d{n} - \log_d{\log_d n} - 6$. Fix an ordering of the vertices. Suppose we are given a not necessarily proper colouring of the first $m$ vertices of $\bar{G}$, using colours 1 up to  $\lfloor n/u \rfloor$. Let $c_i = c_i(m)$ be the number of times colour $i$ is used and let $A_{m+1}$ be the event that the $(m+1)$-th vertex has a neighbour in every colour. We claim that $\Pr(A_{m+1}) = o(1/n)$. Having proved this, we colour the graph greedily. The probability that we need more than $\lfloor n/u \rfloor$ colours is at most $\sum_{m=1}^n\Pr(A_m) = o(1)$, so by Markov, we have \whp\ $\theta(G) = \chi(\bar{G}) \leq n/u$.\vs

To prove our claim, let $B_i = B_i(m)$ be the event that the $(m+1)$-th vertex is joined (in $\bar{G}$,) with an earlier vertex of colour $i$. Then clearly $\Pr(B_i) \leq 1 - p^{c_i}$. Let $Y$ be the number of colours appearing on earlier neighbours of the $(m+1)$-th vertex. Then
\[ \E Y \leq \frac{n}{u} - \sum p^{c_i} \leq \frac{n}{u} - \frac{n}{u}p^{mu/n} \leq \frac{n}{u}\left(1 - p^{u}\right). \]
Let $X = Y - \E Y + \frac{n}{u}(1-p^{u})$ and let $t = c\frac{n}{u}p^{u}$, for some $0 < c < 1$ to be determined later. Then $X$ satisfies the conditions of \hyperlink{Talagrand}{Talagrand's Inequality} with $K = 2$ and $r = 1$. Note that, for $n$ large enough, $0 \leq t \leq \E X$, so
\begin{multline*}
\Pr\left(|X - \E X| \geq c\frac{n}{u}p^{u} + 64 + 32 \sqrt{\frac{n}{u}(1 - p^{u})} \right) \leq \\
4 \exp{\Set{-\frac{c^2 n p^{2u}}{64u}}} \leq 4 \exp{\Set{-\frac{ c^2 \log{n}}{32p^{12} \log{(1/p)} } }}.
\end{multline*}
But $32x^{12} \log{(1/x)} \leq 8/(3e) < 1$ whenever $0 < x < 1$.  Hence, choosing any $c$ with $\sqrt{8/3e} < c < 1$, we deduce that
\[ \Pr\left(|X - \E X| \geq c\frac{n}{u}p^{u} + 64 + 32 \sqrt{\frac{n}{u}(1 - p^{u})} \right) = o(1/n). \]
In particular, since $Y \leq X$,
\[ \Pr\left(Y \geq \frac{n}{u} - c\frac{n}{u}p^{2u} + 64 + 32 \sqrt{\frac{n}{u}} \right) = o(1/n). \]
Since
\[ \frac{n}{u}p^{2u} = \frac{(\log_d{n})^2}{up^{12}} \to \infty,\]
we deduce that (for $n$ large enough,)
\[ \Pr(A_{m+1}) = \Pr(Y \geq \lfloor n/u \rfloor) = o(1/n). \qedhere\]
\end{proof}

\section{Expansion and consequences}\label{S:Expansion}

We now proceed to the expansion properties of random Latin square graphs and to the proof of \autoref{second eigenvalue} on the second eigenvalue of such graphs. In~\cite{Christofides&Markstrom08}, we generalized Hoeffding's inequality, to an inequality where the random variables do not necessarily take real values, but instead take their values in the set of (self-adjoint) operators of a (finite dimensional) Hilbert space. We then used this inequality to give a new proof of the Alon-Roichman theorem. The main tool in the proof of \autoref{second eigenvalue} will be this \hyperref[Operator-Hoeffding]{Operator Hoeffding Inequality}. Before stating the inequality, we need to introduce some more notation.\vs

Let $V$ be a Hilbert space of dimension $d$, let $A(V)$ be the set of self adjoint operators on $V$ and let $P(V)$ be the be the cone of positive operators on $V$, i.e.
\[ P(V) = \{A \in A(V): \text{all eigenvalues of $A$ are nonnegative} \}.\]
This defines a partial order on $A(V)$ by $A \leq B$ iff $B-A \in P(V)$. We denote by $[A,B]$ the set of all $C \in A(V)$ such that $A \leq C \leq B$. We also denote by $\|A\|$ the largest eigenvalue of $A$ in absolute value.\vs

We can now state our \hyperref[Operator-Hoeffding]{Operator Hoeffding Inequality}. We refer the reader to~\cite{Christofides&Markstrom08} for its proof.

\begin{theorem}[\cite{Christofides&Markstrom08} Operator Hoeffding Inequality]\label{Operator-Hoeffding}
Let $V$ be a Hilbert space of dimension $d$ and let $X_i = \E(X|\mathscr{F}_i)$ be a martingale, taking values in $A(V)$, whose difference sequence satisfies $Y_i \in [-\frac{1}{2}I, \frac{1}{2} I]$. Then
\[ \Pr(\|X - \E X \| \geq nh) \leq 2d \exp{ \{-nH(1/2 + h)\} }.\]
\end{theorem}

Note that the case $d = 1$ of this inequality is exactly Hoeffding's inequality.\vs

We now proceed to show that random Latin square graphs have small second eigenvalue and thus good expansion properties.

\begin{proof}[Proof of \autoref{second eigenvalue}]
Let $s_1,\ldots,s_k$ be elements of $[n]$ chosen independently and uniformly at random. For $s \in [n]$ let $L(s)$ be the 0,1 matrix in which $L(s)_{ij}= 1$ if and only if $L_{ij} = s$.  So the normalised adjacency matrix of the multigraph $G$ generated by these elements is $T = \frac{1}{2k}\sum\limits_{i=1}^k{(L(s_i) + L(s_i)^T)}$. Let $B = T - \frac{1}{n}J$, where $J$ is the $n$ by $n$ matrix having `1' in every entry. We claim that $\mu(G) = \| B \|$. Indeed, if $\{v_0,v_1,\ldots,v_{n-1}\}$ is an orthonormal basis of $T$, with each $v_i$ having eigenvalue $\lambda_i$, and $v_0 = \frac{1}{\sqrt{n}}(1,\ldots,1)$, then $Bv_0 = 0$ and $Bv_i = \lambda_i v_i$, so $\mu(G) = \| B \|$ as required. Let $Y_i$ be the operator whose matrix is $\frac{1}{4}\left(L(s_i) + L(s_i)^T - \frac{2}{n}J \right)$. It is easy to check that $X_i = Y_1 + \ldots + Y_i$ is a martingale satisfying the conditions of the theorem. It follows that
\begin{align*}
 \Pr(\mu(G) \geq \varepsilon) &= \Pr\left(\Norm{\frac{1}{k}\sum_{i=1}^k{Y_i}} \geq \frac{\varepsilon}{2} \right) \\
                           &= \Pr \left(\Norm{X - \E X} \geq \frac{\varepsilon k}{2} \right) \\
                           &\leq 2 n \exp{\Set{-kH\left(\frac{1 + \varepsilon}{2}\right)}},
\end{align*}
as required.
\end{proof}

\begin{proof}[Proof of \autoref{vertex-connectivity}]
Let $T$ be a minimial disconnecting set, so $|T| \leq \delta(G) \leq 2k$. Let $U$ be the smallest component of $G \setminus T$ and let $W = V(G) \setminus {(U \cup T)}$. Recalling that $k \leq n/4$, we see that $|W| \geq n/4$. We claim that \whp, $|U| \leq 128 \log{n}$. So let us assume that $|U| > 128\log{n}$ and aim to obtain a contradiction.\vs

Our proof will be very similar to~\cite[Theorem 4.1]{Krivelevich&Sudakov06}. In particular, we will use the edge distribution bound for pseudorandom graphs (see e.g.~\cite[Theorem 2.11]{Krivelevich&Sudakov06}) which implies that for every subsets $A,B$ of the vertex set of $G$ we have
\begin{equation}\label{edge-distribution}
\left|e(A,B) - \frac{2k}{n}|A||B| \right| \leqslant 2k \mu(G) \sqrt{|A||B|}.
\end{equation}

Here $e(A,B)$ denotes the number of edges between $A$ and $B$ counted with multiplicity. Note in particular that this means that edges in $G[A \cap B]$ are counted twice.\vs

Firstly, we deduce from \autoref{second eigenvalue} that \whp\ 

\[
\mu(G) \leq 2\sqrt{\frac{\log{n}}{k}}.
\]

Since $e(U,W) = 0$, it follows from~\eqref{edge-distribution} that 

\[ 
|U| |W| < n\mu(G) \sqrt{|U||W|}
\]

and so
\[ 
|U| < \frac{\mu(G)^2 n^2}{|W|} \leq 4\mu(G)^2 n \leq \frac{16 n \log{n}}{k}.
\]

Using this together with~\eqref{edge-distribution} we get that

\[ 
e(U,U) \leq \frac{2k}{n}|U|^2 + 2k \mu(G) |U| < (32\log{n} + 4 \sqrt{k\log{n}})|U| \leq \frac{k}{2}|U|,
\]

where in the last inequality we used the assumption that $k \geqslant C\log{n}$, for some large enough $C$. It follows that 

\[ e(U,T) = 2k|U| - e(U,U) > \frac{3k}{2}|U|.\]

On the other hand, using~\eqref{edge-distribution} once more, we have

\[ 
e(U,T) \leq \frac{2k}{n}|U||T| + 2k\mu(G) \sqrt{|U||T|} \leq \left(\frac{2|T|}{n} +  4\sqrt{\frac{|T|\log{n}}{k|U|}} \right)k|U| \leq \frac{3k}{2}|U|, 
\]

where we have used the facts that $|T| \leqslant 2k$, $k \leqslant n/2$ and the assumption that $|U| \geqslant 128\log{n}$.\vs

But this is a contradiction as we have proved that $3k|U|/2 \leqslant e(U,T) < 3k|U|/2$. So we may assume that $|U| < 128 \log{n}$.\vs

We now claim that \whp, the following holds: For any 3 distinct vertices $x,y,z$ of $G$, $|(N(x) \cup N(y)) \setminus N(z)| > 128 \log{n}$, where $N(x)$ denotes the neighbourhood of the vertex $x$. Having proved this, it will follow that \whp\ $|U| \leq 2$ and so $|T| \geq \delta(G) - 1$. To see how this follows, observe that if $U$ contains three vertices, say $x,y,z$, then the total number of their neighbours is \whp\ bigger than $128\log{n} + |N(z)|$. In particular, there must be at least $128\log{n} + |N(z)| - |U| > |N(z)|$ vertices outside $U$ which have one of $x,y,z$ as their neighbour. But this implies that $|T| > |N(z)| \geqslant \delta(G)$, contradicting the minimality of $|T|$.\vs

So, let $x,y,z$ be distinct vertices of $G$. Let $s_1,s_2,\ldots,s_k$ be the elements of $S$ chosen uniformly at random and let $X_i = \E (|(N(x) \cup N(y)) \setminus N(z)| | s_1,\ldots,s_i)$. Then $X_0,X_1,\ldots,X_k$ is a martingale with Lipschitz constant 4 and $X_0 \geq k(1 - 2/n)^k$. It follows by the Hoffding-Azuma inequality that
\[ \Pr(|(N(x) \cup N(y)) \setminus N(z)| \leq k(1 - 2/n)^k - t) \leq \exp{\left\{ - \frac{t^2}{2k} \right\}}.\]

Now let $t = \sqrt{8k\log{n}}$ and observe that if $C$ and $n$ are large enough then $128\log{n} \leqslant k(1 - 2/n)^k - t$ and so

\[ 
\Pr(|(N(x) \cup N(y)) \setminus N(z)| \leq 128\log{n}) = O(n^{-4}).
\]
Our claim now follows from the union bound. This completes the proof of the theorem. (It can be checked that $C = 168$ works.)
\end{proof}

We omit the proof of \autoref{edge-connectivity}, as it can be proved using a similar argument as in~\cite[Theorem 4.3]{Krivelevich&Sudakov06}

\begin{proof}[Sketch proof of \autoref{Hamiltonicity}]
Firstly, one needs to check that the result of Krivelevich and Sudakov~\cite{Krivelevich&Sudakov03} mentioned before the statement of the theorem, also holds for $d$-regular multigraphs. We omit the details of this check. Then the result follows directly from \autoref{second eigenvalue}.
\end{proof}

\section{Conclusion and open problems}\label{Final}

We have introduced new models of random graphs arising from Latin squares and studied some of their properties. There is still a lot of research that needs to be done even for many of the properties that we have considered here.\vs

Regarding the clique and independence numbers it would be
interesting to know if the upper bound can be reduced further. In
particular, we believe (but cannot prove) that the 2 in the exponent
can be reduced further. It would be also interesting to know whether
there are examples of random Latin square graphs whose
clique/independence number is significantly larger than $\Theta(\log{n}
\log{\log{n}})$. It looks plausible that this is not the case.\vs

Similar remarks hold for the lower bound on the chromatic and clique cover numbers. Any improvement on the upper bound of the independence/clique numbers would give a corresponding improvement on the chromatic/clique cover numbers but it might be possible (or even easier) to get such improvements directly.\vs

Another interesting question which we have not been able to answer so far is the determination of the Hadwiger number of random Latin square graphs, i.e.~the largest integer $k$ such that the graph can be contracted into a $K_k$. We do not even know, for $p = 1/2$ say, whether this number depends on the sequence of Latin squares or not.\vs

We have not studied at all the girth of random Latin square graphs. The reason is that it depends a lot on the structure of the Latin squares chosen. For example, almost every $G \in \mathscr{G}(\mathbb{Z}_3^m,p)$ has \whp\ girth 3, provided $pn \to \infty$, where $n = 3^m$. On the other hand, we claim that almost every $G \in \mathscr{G}(\mathbb{Z}_2^m,p)$ has \whp\ girth strictly greater than 3 provided that $pn^{2/3} \to 0$, where $n = 2^m$. Indeed, the expected number of triangles containing a fixed vertex $x$ is $\binom{n-1}{2}p^3$ which tends to 0. By Markov's inequality $x$ is \whp\ not contained in any triangle. But since the graph is vertex transitive, our claim follows.\vs

The expansion properties of random Latin squares imply that almost every $G \in \mathscr{G}(L_n, c\log_2{n})$, with $c > 1$, has logarithmic diameter. An interesting question here is the threshold for the diameter becoming equal to 2. It turns out that there are constants $c_1$ and $c_2$ such that if $p < c_1 \sqrt{\log{n}/n}$, then almost every $G \in \mathscr{G}(L_n,p)$ has diameter greater than 2, while if $p > c_2 \sqrt{\log{n}/n}$, then almost every $G \in \mathscr{G}(L_n,p)$ has diameter less than or equal to 2. The values of $c_1$ and $c_2$ depend on the sequence of Latin squares chosen. Our results regarding the diameter will appear in a forthcoming paper~\cite{Christofides&Markstrom+}.

\section*{Acknowledgements}

This work started when the second author was visiting the Department of Pure Mathematics and Mathematical Statistics of the University of Cambridge and continued when the first author was visiting the Deparment of Mathematics and Mathematical Statistics of \Umea\ University. The authors would like to thank both departments for their hospitality.\vs

The first author would like to thank the Engineering and Physical Sciences Research Council and the Cambridge Commonwealth Trust which supported him while he was at the University of Cambridge and also the Deparment of Mathematics and Mathematical Statistics of \Umea\ University for the grant making his visit there possible. The second author would like to thank the Royal Physiographic Society in Lund for the grant making his visit to the Univesity of Cambridge possible.

\bigskip

\noindent
{\footnotesize
Demetres Christofides, Institute for Theoretical Computer Science, Faculty of Mathematics and Physics, \Malostranske\ \Namesti\ 25, 188 00 Prague, Czech Republic, \href{mailto:christofidesdemetres@gmail.com}{\tt christofidesdemetres@gmail.com}

\smallskip

\noindent Klas \Markstrom, Department of Mathematics and Mathematical Statistics, \Umea\ University, 90187 \Umea, Sweden, \href{mailto:klas.markstrom@math.umu.se}{\tt klas.markstrom@math.umu.se}
}

\end{document}